\documentclass{article}%
\usepackage{amsmath}
\usepackage{amsfonts}
\usepackage{amssymb}
\usepackage{graphicx}%
\providecommand{\U}[1]{\protect\rule{.1in}{.1in}}
\newtheorem{theorem}{Theorem}[section]

\newtheorem{conjecture}[theorem]{Conjecture}
\newtheorem{corollary}[theorem]{Corollary}

\newtheorem{lemma}[theorem]{Lemma}

\newtheorem{problem}[theorem]{Problem}
\newtheorem{proposition}[theorem]{Proposition}

\newenvironment{proof}[1][Proof]{\noindent\textbf{#1.} }{\ \rule{0.5em}{0.5em}}
\begin{document}

\author{Zakir Deniz\\Duzce University, Duzce, Turkey\\zakirdeniz@duzce.edu.tr
\and Vadim E. Levit\\Ariel University, Israel\\levitv@ariel.ac.il
\and Eugen Mandrescu\\Holon Institute of Technology, Israel\\eugen\_m@hit.ac.il}
\title{On graphs admitting two disjoint maximum independent sets}
\date{}
\maketitle

\begin{abstract}
An independent set $S$ is \textit{maximal} if it is not a proper subset of an
independent set, while $S$ is \textit{maximum} if it has a maximum size. The
problem of whether a graph has a pair of disjoint maximal independent sets was
introduced by Berge in \ early 70's. The class of graphs for which every
induced subgraph admits two disjoint maximal independent sets was
characterized in (Schaudt, 2015). It is known that deciding whether a graph
has two disjoint maximal independent sets is a \textbf{NP}-complete problem
(Henning \textit{et al}., 2009).

In this paper, we are focused on finding conditions ensuring the existence of
two disjoint maximum independent sets.

\textbf{Keywords}:\ maximum independent set, shedding vertex,
K\"{o}nig-Egerv\`{a}ry graph, unicyclic graph, well-covered graph, corona of graphs.

\end{abstract}

\section{Introduction}

Throughout this paper $G=(V,E)$ is a finite, undirected, loopless graph
without multiple edges, with vertex set $V=V(G)$ of cardinality $\left\vert
V\left(  G\right)  \right\vert =n\left(  G\right)  $, and edge set $E=E(G)$ of
size $\left\vert E\left(  G\right)  \right\vert =m\left(  G\right)  $. If
$X\subset V$, then $G[X]$ is the graph of $G$ induced by $X$. By $G-U$ we mean
the subgraph $G[V-U]$, if $U\subset V(G)$. We also denote by $G-F$ the
subgraph of $G$ obtained by deleting the edges of $F$, for $F\subset E(G)$,
and we write shortly $G-e$, whenever $F$ $=\{e\}$.

The \textit{neighborhood} $N(v)$ of $v\in V\left(  G\right)  $ is the set
$\{w:w\in V\left(  G\right)  $ \textit{and} $vw\in E\left(  G\right)  \}$,
while the \textit{closed neighborhood} $N[v]$\ of $v$ is the set
$N(v)\cup\{v\}$. Let $\deg\left(  v\right)  =\left\vert N(v)\right\vert $. If
$\deg\left(  v\right)  =1$, then $v$ is a \textit{leaf}, and $\mathrm{Leaf}%
\left(  G\right)  $ is the set containing all the leaves.

The \textit{neighborhood} $N(A)$ of $A\subseteq V\left(  G\right)  $ is
$\{v\in V\left(  G\right)  :N(v)\cap A\neq\emptyset\}$, and $N[A]=N(A)\cup A$.
We may also use $N_{G}(v),N_{G}\left[  v\right]  ,N_{G}(A)$ and $N_{G}\left[
A\right]  $, when referring to neighborhoods in a graph $G$.

$C_{n},K_{n},P_{n},K_{p,q}$ denote respectively, the cycle on $n\geq3$
vertices, the complete graph on $n\geq1$ vertices, the path on $n\geq1$
vertices, and the complete bipartite graph on $p+q$ vertices, where $p,q\geq1$.

A \textit{matching} is a set $M$ of pairwise non-incident edges of $G$, and by
$V\left(  M\right)  $ we mean the vertices covered by $M$. If $V\left(
M\right)  =V\left(  G\right)  $, then $M$ is a \textit{perfect matching}. The
size of a largest matching is denoted by $\mu\left(  G\right)  $. If every
vertex of a set $A$ is an endpoint of an edge $e\in M$, while the other
endpoint of $e$ belongs to some set $B$, disjoint from $A$, we say that $M$ is
a \textit{matching} \textit{from} $A$ \textit{into} $B$, or $A$ is
\textit{matched into} $B$ by $M$. In other words, $M$ may be interpreted as an
injection from the set $A$ into the set $B$.

The \textit{disjoint union} $G_{1}\cup G_{2}$ of the graphs $G_{1}$ and
$G_{2}$ with $V(G_{1})\cap V(G_{2})=$ $\emptyset$ is the graph having
$V(G_{1})\cup V(G_{2})$ and $E(G_{1})\cup E(G_{2})$ as a vertex set and an
edge set, respectively. In particular, $qG$ denotes the disjoint union of $q$
$\geq2$ copies of the graph $G$.

A set $S\subseteq V(G)$ is \textit{independent} if no two vertices from $S$
are adjacent, and by $\mathrm{Ind}(G)$ we mean the family of all the
independent sets of $G$. An independent set $A$ is \textit{maximal} if
$A\cup\left\{  v\right\}  $ is not independent, for every $v\in V\left(
G\right)  -A$. An independent set of maximum size is a \textit{maximum
independent set} of $G$, and $\alpha(G)=\max\{\left\vert S\right\vert
:S\in\mathrm{Ind}(G)\}$.

\begin{theorem}
\cite{Berge1982}\label{th6} In a graph $G$, an independent set $S$ is maximum
if and only if every independent set disjoint from $S$ can be matched into $S$.
\end{theorem}

Let \textrm{core}$(G)=%
{\displaystyle\bigcap}
\{S:S\in\Omega(G)\}$, where $\Omega(G)$ denotes the family of all maximum
independent sets \cite{LevMan2002}.

\begin{theorem}
\cite{LevMan2001}\label{th8} A connected bipartite graph $G$ has a perfect
matching if and only if $\mathrm{core}\left(  G\right)  =\emptyset$.
\end{theorem}

If $\alpha(G)+\mu(G)=n\left(  G\right)  $, then $G$ is \textit{a
K\"{o}nig-Egerv\'{a}ry graph }\cite{Deming,Sterboul}. It is known that every
bipartite graph is a K\"{o}nig-Egerv\'{a}ry graph as well.

Let $v\in V\left(  G\right)  $. If for every independent set $S$ of
$G-N\left[  v\right]  $, there exists some $u\in N\left(  v\right)  $ such
that $S\cup\left\{  u\right\}  $ is independent, then $v$ is a
\textit{shedding vertex} of $G$ \cite{Woodroofe2009}. Clearly, no isolated
vertex may be a shedding vertex. On the other hand, every vertex of degree
$n\left(  G\right)  -1$ is a shedding vertex. Let $\mathrm{Shed}\left(
G\right)  $ denote the set of all shedding vertices. For instance,
$\mathrm{Shed}\left(  K_{1}\right)  =\emptyset$, while $\mathrm{Shed}\left(
K_{n}\right)  =V\left(  K_{n}\right)  $ for every $n\geq2$.

A vertex $v$ of a graph $G$ is \textit{simplicial} if the induced subgraph of
$G$ on the set $N[v]$ is a complete graph and this complete graph is called a
simplex of $G$. Clearly, every leaf is a simplicial vertex. Let $\mathrm{Simp}%
\left(  G\right)  $ denote the set of all simplicial vertices.

\begin{proposition}
\label{prop1}\cite{Woodroofe2009} If $v\in\mathrm{Simp}\left(  G\right)  $,
then $N\left(  v\right)  \subseteq\mathrm{Shed}\left(  G\right)  $.
\end{proposition}

A graph $G$ is said to be \textit{simplicial} if every vertex of $G$ belongs
to a simplex of $G$. By Proposition \ref{prop1}, if every simplex of a
simplicial graph $G$ contains two simplicial vertices at least, then
$\mathrm{Shed}\left(  G\right)  =V\left(  G\right)  $. The converse is not
necessarily true. For instance, $C_{5}$ has no simplicial vertex, while
$\mathrm{Shed}\left(  C_{5}\right)  =V\left(  C_{5}\right)  $.

A vertex $v\in V\left(  G\right)  $ is \textit{codominated} if there is
another vertex $u\in V\left(  G\right)  $ such that $N\left[  u\right]
\subseteq$ $N\left[  v\right]  $. In such a case, we say that $v$ is
codominated by $u$. For instance, consider the graphs $G_{1}$ and $G_{2}$ from
Figure \ref{fig2}: $x,z\in\mathrm{Shed}(G_{1})$, and both vertices are
codominated, while $w\in\mathrm{Shed}(G_{2})$ and $w$ is not codominated.

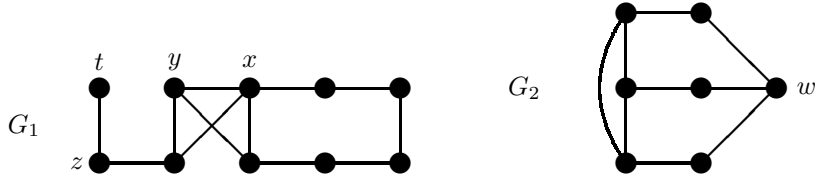
\begin{figure}[h]
\setlength{\unitlength}{1cm}\begin{picture}(5,2)
\thicklines
\multiput(2,0)(1,0){5}{\circle*{0.29}}
\multiput(2,1)(1,0){5}{\circle*{0.29}}
\put(2,0){\line(1,0){1}}
\put(2,0){\line(0,1){1}}
\put(3,1){\line(1,0){1}}
\put(3,1){\line(1,-1){1}}
\put(3,0){\line(0,1){1}}
\put(3,0){\line(1,1){1}}
\put(4,0){\line(0,1){1}}
\put(4,0){\line(1,0){2}}
\put(4,1){\line(1,0){2}}
\put(6,0){\line(0,1){1}}
\put(2,1.35){\makebox(0,0){$t$}}
\put(1.7,0){\makebox(0,0){$z$}}
\put(3,1.35){\makebox(0,0){$y$}}
\put(4,1.35){\makebox(0,0){$x$}}
\put(1,0.5){\makebox(0,0){$G_{1}$}}
\multiput(9,0)(1,0){2}{\circle*{0.29}}
\multiput(9,1)(1,0){3}{\circle*{0.29}}
\multiput(9,2)(1,0){2}{\circle*{0.29}}
\put(9,0){\line(1,0){1}}
\put(9,0){\line(0,1){2}}
\put(9,1){\line(1,0){2}}
\put(9,2){\line(1,0){1}}
\put(10,0){\line(1,1){1}}
\put(10,2){\line(1,-1){1}}
\qbezier(9,0)(8.3,1)(9,2)
\put(11.4,1){\makebox(0,0){$w$}}
\put(7.65,1){\makebox(0,0){$G_{2}$}}
\end{picture}\caption{$N[t]\subseteq N[z]$ and $N[y]\subseteq N[x]$.}%
\label{fig2}%
\end{figure}

\begin{lemma}
\label{lem11}\cite{Biyi} Every codominated vertex is a shedding vertex as
well. Moreover, in a bipartite graph, each shedding vertex is also a
codominated vertex, and if $x$ is codominated by $y$, then $y$ is a leaf.
\end{lemma}

\begin{theorem}
\label{th4}\cite{CaCrRey2016} If $v\in\mathrm{Shed}\left(  G\right)  $, then
one of the following hold:

\emph{(i)} there exists $u\in N\left(  v\right)  $, such that $N\left[
u\right]  \subseteq N\left[  v\right]  $, i.e., $v$ is a codominated vertex;

\emph{(ii)} $v$ belongs to some $5$-cycle.
\end{theorem}

A graph is \textit{well-covered} if all its maximal independent sets are also
maximum \cite{plum}. If $G$ is well-covered, without isolated vertices, and
$n\left(  G\right)  =2\alpha\left(  G\right)  $, then $G$ is a \textit{very
well-covered graph} \cite{Favaron1982}. The only well-covered cycles are
$C_{3}$, $C_{4}$, $C_{5}$ and $C_{7}$, while $C_{4}$ is the unique very
well-covered cycle.

\begin{theorem}
\label{th77}\cite{LevMan2008}\emph{ }$G$ is very well-covered if and only if
$G$ is a well-covered K\"{o}nig-Egerv\`{a}ry graph.
\end{theorem}

Let $\mathcal{H}=\{H_{v}:v\in V(G)\}$ be a family of graphs indexed by the
vertex set of a graph $G$. The corona $G\circ\mathcal{H}$ of $G$ and
$\mathcal{H}$ is the disjoint union of $G$ and $H_{v},v\in V(G)$, with
additional edges joining each vertex $v\in V(G)$ to all the vertices of
$H_{v}$. If $H_{v}=H$ for every $v\in V(G)$, then we denote $G\circ H$ instead
of $G\circ\mathcal{H}$ \cite{FruchtHarary}. It is known that $G\circ
\mathcal{H}$ is well-covered if and only if each $H_{v},v\in V(G)$, is a
complete graph \cite{Topp}.

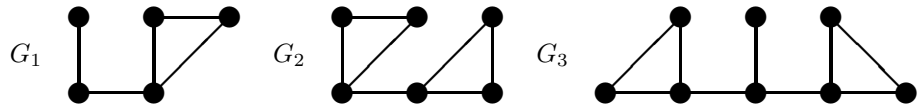
\begin{figure}[h]
\setlength{\unitlength}{1cm}\begin{picture}(5,1.25)\thicklines
\multiput(1.5,0)(1,0){2}{\circle*{0.29}}
\multiput(1.5,1)(1,0){3}{\circle*{0.29}}
\put(1.5,0){\line(1,0){1}}
\put(2.5,1){\line(1,0){1}}
\put(1.5,0){\line(0,1){1}}
\put(2.5,0){\line(0,1){1}}
\put(2.5,0){\line(1,1){1}}
\put(0.8,0.5){\makebox(0,0){$G_{1}$}}
\multiput(5,0)(1,0){3}{\circle*{0.29}}
\multiput(5,1)(1,0){3}{\circle*{0.29}}
\put(5,0){\line(1,0){2}}
\put(5,0){\line(1,1){1}}
\put(5,1){\line(1,0){1}}
\put(5,0){\line(0,1){1}}
\put(6,0){\line(1,1){1}}
\put(7,0){\line(0,1){1}}
\put(4.3,0.5){\makebox(0,0){$G_{2}$}}
\multiput(8.5,0)(1,0){5}{\circle*{0.29}}
\multiput(9.5,1)(1,0){3}{\circle*{0.29}}
\put(8.5,0){\line(1,0){4}}
\put(8.5,0){\line(1,1){1}}
\put(9.5,0){\line(0,1){1}}
\put(10.5,0){\line(0,1){1}}
\put(11.5,0){\line(0,1){1}}
\put(11.5,1){\line(1,-1){1}}
\put(7.8,0.5){\makebox(0,0){$G_{3}$}}
\end{picture}\caption{$G_{1}=P_{2}\circ\left\{  K_{1},K_{2}\right\}  $,
$G_{2}=P_{2}\circ K_{2}$, $G_{3}=P_{3}\circ\left\{  K_{2},K_{1},K_{2}\right\}
$.}%
\label{fig3}%
\end{figure}

Recall that the \textit{girth} of a graph $G$ is the length of a shortest
cycle contained in $G$, and it is defined as the infinity for every forest.

\begin{theorem}
\label{th7}\emph{(i)} \cite{FinHarNow} Let $G$ be a connected graph of girth
$\geq6$, which is isomorphic to neither $C_{7}$ nor $K_{1}$. Then $G$ is
well-covered if and only if $G=H\circ K_{1}$ for some graph $H$.

\emph{(ii)} \cite{LevMan2007} Let $G$ be a connected graph of girth $\geq5$.
Then $G$ is very well-covered if and only if $G=H\circ K_{1}$ for some graph
$H$.
\end{theorem}

The graph $P_{3}$ has two disjoint maximal independent sets, while $C_{4}$ has
even two disjoint maximum independent sets. On the other hand, the graph
$C_{5}\circ K_{1}$ has no pair of disjoint maximal independent sets. The
graphs from Figure \ref{fig2} have pairs of disjoint maximal (non-maximum)
independent sets, while the graphs from Figure \ref{fig3} have pairs of
disjoint maximum independent sets.

The research on the graphs admitting two disjoint maximal independent sets has
its roots in \cite{Berge1973,Payan1974}. Further, this topic was studied in
\cite{CoHe1976,EHP1982,Fieux2017,HLR2009,Payan1978,Schaudt2015}. A
constructive characterization of trees that have two disjoint maximal
independent sets of minimum size may be found in \cite{HayHen2005}.

By definition, for well-covered graphs to find out two disjoint maximal
independent sets is the same as to detect two maximum independent sets.

\begin{theorem}
\cite{Schaudt2015} Let $G$ be a well-covered graph without isolated vertices.
If $G$ does not contain $C_{2k+1}\circ K_{1}$ as an induced subgraph for
$k\geq1$, then $G$ has two disjoint maximum independent sets.
\end{theorem}

The same problem in line graphs is about two disjoint maximum matchings.

\begin{theorem}
\cite{LevMan2001} A bipartite graph has two disjoint perfect matchings if and
only if it has a partition of its vertex set comprising of a family of simple cycles.
\end{theorem}

The most well-known subclass of graphs with two disjoint maximum independent
sets is the family of $W_{2}$-graphs. Recall that a graph $G$ belongs to
$W_{2}$ if every two pairwise disjoint independent sets are included in two
pairwise disjoint maximum independent sets \cite{LevMan2016,Staples}.

In this paper, we concentrate on graphs admitting two disjoint maximum
independent sets.

\section{General graphs}

\begin{theorem}
\label{prop:two disjoint-matching}For every graph $G$, the following
assertions are equivalent:

\emph{(i)} $G$ has two disjoint maximum independent sets;

\emph{(ii)} there exists a maximum independent set $S$ such that
$\alpha(G-S)=\alpha(G)$;

\emph{(iii) }there exists a matching $M$ of size $\alpha(G)$ such that
$G[V(M)]$ is a bipartite graph;

\emph{(iv)} $G$\ has an induced bipartite subgraph of order $2\alpha(G)$;

\emph{(v)} there exists a set $A\subset V(G)$ such that $G-A$ is a bipartite
graph having a perfect matching of size $\alpha(G)$.
\end{theorem}

\begin{proof}
\emph{(i)} $\Leftrightarrow$ \emph{(ii)} It is clear.

\emph{(i)} $\Rightarrow$ \emph{(iii) }Assume that $S_{1},S_{2}$ are two
disjoint maximum independent sets in $G$. Then $G\left[  S_{1}\cup
S_{2}\right]  $ is a bipartite subgraph in $G$, having, by Theorem \ref{th6},
a perfect matching $M$ of size $\alpha(G)$. Clearly, $G[V(M)]=G\left[
S_{1}\cup S_{2}\right]  $.

\emph{(iii)} $\Rightarrow$ \emph{(i) }Suppose that there exists a matching $M$
of size $\alpha(G)$ such that $G[V(M)]$ is a bipartite graph. Consequently,
the bipartition $\left\{  A,B\right\}  $ of $G[V(M)]$\ provides two disjoint
maximum independent sets, namely $A$ and $B$.

\emph{(iii)} $\Rightarrow$ \emph{(iv) }It is clear.

\emph{(iv)} $\Rightarrow$ \emph{(iii) }Let $H=(A,B,U)$ be an induced bipartite
subgraph of order $2\alpha(G)$. Hence, $\left\vert A\right\vert =\left\vert
B\right\vert =\alpha(G)$, since $A$ and $B$ are independent sets in $G$.
Moreover, Theorem \ref{th6} ensures a perfect matching of size $\alpha(G)$ in
$H$.

\emph{(iii)} $\Leftrightarrow$ \emph{(v) }It is clear.
\end{proof}

\begin{corollary}
\label{cor3}If $G$ has two disjoint maximum independent sets, then $\mu\left(
G\right)  \geq\alpha\left(  G\right)  $.
\end{corollary}

\begin{proof}
Let $S_{1},S_{2}$ are two disjoint maximum independent sets in $G$. By Theorem
\ref{th6}, there exists some matching from $S_{1}$ into $S_{2}$. Since
$\left\vert S_{1}\right\vert =\left\vert S_{2}\right\vert =\alpha\left(
G\right)  $, we infer that $\mu\left(  G\right)  \geq\alpha\left(  G\right)  $.
\end{proof}

\begin{theorem}
\label{th1}Let $S\in\mathrm{Ind}\left(  G\right)  $ and $S\subseteq
\mathrm{Shed}(G)$, then

\emph{(i) }the number of independent sets of size $\left\vert S\right\vert $
in $G$ is greater or equal to $2^{\left\vert S\right\vert }$;

\emph{(ii) }there exist some maximal independent set $U$ disjoint from $S$,
and a matching from $S$ into $U$.
\end{theorem}

\begin{proof}
Let $A=\{x_{i_{1}},x_{i_{2}},...,x_{i_{k}}\}\subseteq S$. Now, we are
constructing an independent set $I_{A}=\left(  S-A\right)  \cup B_{A}$ such
that $S\cap B_{A}=\emptyset$ and $\left\vert A\right\vert =\left\vert
B_{A}\right\vert $.

We start by taking a vertex $y_{i_{1}}\in N_{G}(x_{i_{1}})$, such that
$I_{i_{1}}=\left(  S-A\right)  \cup\{y_{i_{1}}\}$ is independent. This is
possible, because $x_{i_{1}}$ is a shedding vertex.

Further, let us consider the vertex $x_{i_{j}}$ for each $2\leq j\leq k$.
Then, there exists some $y_{i_{j}}\in N_{G}(x_{i_{j}})$, such that $y_{i_{j}}$
is not adjacent to any vertex of $I_{i_{j-1}}$, since $x_{i_{j}}$ is shedding
and $I_{i_{j-1}}$ is independent. Hence, the set $I_{i_{j}}=I_{i_{j-1}}%
\cup\{y_{i_{j}}\}$ is independent.

Finally, $I_{A}=I_{i_{k}}=\left(  S-A\right)  \cup B_{A}$, where
$B_{A}=\{y_{i_{1}},y_{i_{2}},...,y_{i_{k}}\}$. Clearly, $S\cap B_{A}$
$=\emptyset$, and $\{x_{i_{1}}y_{i_{1}},x_{i_{2}}y_{i_{2}},...,x_{i_{k}%
}y_{i_{k}}\}$ is a matching from $A$ into $B_{A}\subseteq I_{A}$.

Suppose $A_{1},A_{2}\subseteq S$. If $A_{1}\neq A_{2}$, then $I_{A_{1}%
}=\left(  S-A_{1}\right)  \cup B_{A_{1}}\neq\left(  S-A_{2}\right)  \cup
B_{A_{2}}=I_{A_{2}}$, since $S-A_{1}\neq S-A_{2}$. In other words, every
subset of $S$ produces an independent set of the same size, and all these sets
are different. Thus the graph $G$ has $2^{\left\vert S\right\vert }$
independent sets of cardinality $\left\vert S\right\vert $, at least.

If $A=S$, then $I_{S}$ is disjoint from $S$. To complete the proof, one has
just to enlarge $I_{S}$ to a maximal independent set, say $U$. The sets
$U$\ and $S$ are disjoint, since there is a matching from $S$ into
$I_{S}\subseteq U$.
\end{proof}

\begin{corollary}
If $G$ has a maximal independent set $S$ such that $S\subseteq\mathrm{Shed}%
(G)$, then there exists a maximal independent set $U$ disjoint from $S$ such
that $\left\vert S\right\vert \leq\left\vert U\right\vert $.
\end{corollary}

\begin{corollary}
\label{cor5}If $G$ has a maximum independent set $S$ such that $S\subseteq
\mathrm{Shed}(G)$, then $\left\vert \Omega(G)\right\vert \geq2^{\alpha\left(
G\right)  }$, while some $I\in\Omega(G)$ is disjoint from $S$.
\end{corollary}

The friendship graph $F_{q}=K_{1}\circ qK_{2},q\geq2$ shows that
$2^{\alpha\left(  G\right)  }$ is a tight lower bound for $\left\vert
\Omega(G)\right\vert $ in graphs with a maximum independent set consisting of
only shedding vertices.

Combining Theorem \ref{th1} and Corollary \ref{cor3}, we deduce the following.

\begin{corollary}
If $G$ has a \ maximum independent set $S\subseteq\mathrm{Shed}(G)$, then
$\mu\left(  G\right)  \geq\alpha\left(  G\right)  $.
\end{corollary}

Notice that each graph from Figure \ref{fig3}\ has a maximum independent set
containing only shedding vertices, and hence, by Theorem \ref{th1}, each one
has two disjoint maximum independent sets.

\begin{corollary}
If $p\geq2$, then $G\circ K_{p}$ has two disjoint maximum independent sets.
\end{corollary}

\begin{proof}
Since $p\geq2$, Proposition \ref{prop1} implies that $\mathrm{Shed}%
(G)=V\left(  G\right)  $. Further, the conclusion follows according to
Corollary \ref{cor5}.
\end{proof}

It is worth mentioning that if $G$ has a pair of disjoint maximum independent
sets, it may have $\mathrm{Shed}\left(  G\right)  =\emptyset$; e.g.,
$G=K_{n,n}$ for $n\geq2$.

\section{K\"{o}nig-Egerv\`{a}ry graphs}

\begin{theorem}
\label{th9}$G$ is a K\"{o}nig-Egerv\`{a}ry graph with two disjoint maximum
independent sets if and only if $G$ is a bipartite graph having a perfect matching.
\end{theorem}

\begin{proof}
Let $S_{1},S_{2}\in\Omega(G)$ and $S_{1}\cap S_{2}=\emptyset$. Since $G$ is a
K\"{o}nig-Egerv\'{a}ry graph and\textit{ }$S_{1}\subseteq V\left(  G\right)
-S_{2}$, we get
\[
\alpha\left(  G\right)  =\left\vert S_{1}\right\vert \leq\left\vert V\left(
G\right)  -S_{2}\right\vert =\left\vert V\left(  G\right)  \right\vert
-\alpha\left(  G\right)  =\mu\left(  G\right)  \leq\alpha\left(  G\right)  .
\]
It follows that $S_{1}=V\left(  G\right)  -S_{2}$ and $\mu\left(  G\right)
=\alpha\left(  G\right)  $. Hence, $G=\left(  S_{1},S_{2},E\left(  G\right)
\right)  $\ is a bipartite graph with a perfect matching.

The converse is evident.
\end{proof}

\begin{corollary}
\label{cor4}If $G$ is a very well-covered graph having two disjoint maximum
independent sets, then $G$ is a bipartite graph with a perfect matching.
\end{corollary}

\begin{proof}
By Theorem \ref{th77}, $G$ is a K\"{o}nig-Egerv\`{a}ry graph. Further,
according to Theorem \ref{th9}, $G$ is a bipartite graph with a perfect matching.
\end{proof}

The converse of Corollary \ref{cor4} is not true; e.g., $G=C_{6}$. It is worth
mentioning that $C_{5}$ is well-covered, non-bipartite, and has some pairs of
disjoint maximum independent sets.

\begin{corollary}
\cite{LevMan2016}\label{cor2} The corona $H\circ K_{1}$ has two disjoint
maximum independent sets if and only if $H$ is a bipartite graph.
\end{corollary}

\begin{proof}
By Corollary \ref{cor4}, $H\circ K_{1}$ must be bipartite, because it is a
very well-covered graph with two disjoint maximum independent sets. Hence, $H$
itself must be bipartite as a subgraph of $H\circ K_{1}$.

Conversely, $H\circ K_{1}$ is a bipartite graph, because $H$ is bipartite.
Clearly, $H\circ K_{1}$ has a perfect matching. Therefore, $H\circ K_{1}$ has
two disjoint maximum independent sets, in accordance with Theorem \ref{th9}.
\end{proof}

Evidently, $C_{5}$ and $C_{7}$ are well-covered and they both have disjoint
maximum independent sets.

\begin{corollary}
\label{cor1}Let $G$ be a well-covered graph of girth $\geq6$ with $K_{1}\neq
G\neq C_{7}$,\ or $G$ be a very well-covered graph of girth $\geq5$. Then $G$
has two disjoint maximum independent sets if and only if $G$ is bipartite.
\end{corollary}

\begin{proof}
According to Theorem \ref{th7}\emph{(i)}, \emph{(ii)}, $G$ must be under the
form $G=H\circ K_{1}$. Now, the result follows by Corollary \ref{cor2}.
\end{proof}

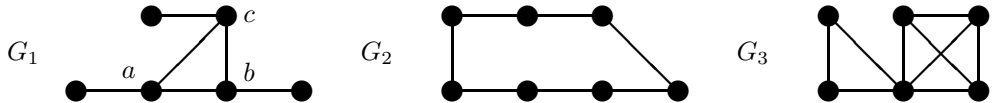
\begin{figure}[h]
\setlength{\unitlength}{1cm}\begin{picture}(5,1.25)\thicklines
\multiput(1,0)(1,0){4}{\circle*{0.29}}
\multiput(2,1)(1,0){2}{\circle*{0.29}}
\put(1,0){\line(1,0){3}}
\put(2,1){\line(1,0){1}}
\put(2,0){\line(1,1){1}}
\put(3,0){\line(0,1){1}}
\put(1.7,0.25){\makebox(0,0){$a$}}
\put(3.3,0.25){\makebox(0,0){$b$}}
\put(3.3,1){\makebox(0,0){$c$}}
\put(0.3,0.5){\makebox(0,0){$G_{1}$}}
\multiput(6,0)(1,0){4}{\circle*{0.29}}
\multiput(6,1)(1,0){3}{\circle*{0.29}}
\put(6,0){\line(1,0){3}}
\put(6,0){\line(0,1){1}}
\put(6,1){\line(1,0){2}}
\put(8,1){\line(1,-1){1}}
\put(5,0.5){\makebox(0,0){$G_{2}$}}
\multiput(11,0)(1,0){3}{\circle*{0.29}}
\multiput(11,1)(1,0){3}{\circle*{0.29}}
\put(11,0){\line(1,0){2}}
\put(11,0){\line(0,1){1}}
\put(11,1){\line(1,-1){1}}
\put(12,0){\line(1,1){1}}
\put(12,0){\line(0,1){1}}
\put(12,1){\line(1,0){1}}
\put(12,1){\line(1,-1){1}}
\put(13,0){\line(0,1){1}}
\put(10,0.5){\makebox(0,0){$G_{3}$}}
\end{picture}\caption{$\mathrm{Shed}\left(  G_{1}\right)  =\left\{
a,b,c\right\}  $, $\mathrm{Shed}\left(  G_{2}\right)  =\emptyset$, while
$\mathrm{Shed}\left(  G_{3}\right)  =V\left(  G_{3}\right)  $.}%
\label{fig4444}%
\end{figure}

Clearly, the graphs $G_{1},G_{2}$ from Figure \ref{fig4444} are well-covered.
The graph $G_{1}$ has no pair of disjoint maximum independent sets, while
$G_{2}$ has such pairs. The graph $G_{3}$ from Figure \ref{fig4444} is not
even well-covered, but it has some pairs of disjoint maximum independent sets.

It is known that $G$ is a K\"{o}nig-Egerv\`{a}ry graph if and only if every
maximum matching matches $V\left(  G\right)  -S$ into $S$, for each
$S\in\Omega(G)$ \cite{LevMan2013}.

\begin{theorem}
\label{lem1}If $G$ is a K\"{o}nig-Egerv\`{a}ry graph, then $\left\vert
\Omega(G)\right\vert \leq2^{\alpha\left(  G\right)  }$. Moreover, the equality
$\left\vert \Omega(G)\right\vert =2^{\alpha\left(  G\right)  }$ holds if and
only if $G=\alpha\left(  G\right)  K_{2}$.
\end{theorem}

\begin{proof}
Let $S\in\Omega(G)$ and $M_{G}$ be a maximum matching of $G$. Then $\left\vert
V\left(  G\right)  -S\right\vert =\left\vert M_{G}\right\vert =\mu\left(
G\right)  $ and each maximum matching of $G$ matches $V\left(  G\right)  -S$
into $S$, since $G$ is a K\"{o}nig-Egerv\`{a}ry graph. Thus every maximum
independent set different from $S$ must contain vertices belonging to
$V\left(  G\right)  -S$. Let us define a graph $H$ as follows: $V\left(
H\right)  =S\cup\left(  V\left(  G\right)  -S\right)  \cup A$, where $A$ is
comprised of $\left\vert S\right\vert -\left\vert V\left(  G\right)
-S\right\vert $ new vertices, while $E\left(  H\right)  =M$, where $M$ is a
perfect matching that matches $S$ into $\left(  V\left(  G\right)  -S\right)
\cup A$ and $M_{G}\subseteq M$. In the other words, $H=\left\vert S\right\vert
K_{2}=\alpha\left(  G\right)  K_{2}$. Then $\Omega(G)\subseteq\Omega(H)$, and
hence, we infer that $\left\vert \Omega(G)\right\vert \leq\left\vert
\Omega(H)\right\vert =2^{\alpha\left(  G\right)  }$, as required.

Clearly, $H$ is well-covered. Consequently, every vertex of $H$ is contained
in some maximum independent set, and adding an edge to $E(H)$ reduces the
number of maximum independent sets.

Suppose $\left\vert \Omega(G)\right\vert =2^{\alpha\left(  G\right)  }$. Then
$\Omega(G)=\Omega(H)$, since $\Omega(G)\subseteq\Omega(H)$.

First, $V(G)=V(H)$, because, otherwise, if there exists some vertex $x\in
V\left(  H\right)  -V\left(  G\right)  $, then each maximum independent set of
$H$ containing $x$ does not appear in $\Omega(G)$, in contradiction with
$\Omega(G)=\Omega(H)$.

Second, $E(G)=E(H)$, since otherwise, if there is an edge $xy\in E(G)-E(H)$,
then each maximum independent set of $H$ containing $\left\{  x,y\right\}  $
does not appear in $\Omega(G)$, in contradiction with $\Omega(G)=\Omega(H)$.

If $G=\alpha\left(  G\right)  K_{2}$, then clearly, $\left\vert \Omega
(G)\right\vert =2^{\alpha\left(  G\right)  }$.
\end{proof}

\begin{theorem}
\label{th2}For a K\"{o}nig-Egerv\`{a}ry graph $G$, the following assertions
are equivalent:

\emph{(i}) there is a maximum independent set included in $\mathrm{Shed}(G)$;

\emph{(ii) }$\left\vert \Omega(G)\right\vert =2^{\alpha\left(  G\right)  }$;

\emph{(iii) }$\mathrm{Shed}(G)=V\left(  G\right)  $;

\emph{(iv)} every maximum independent set is included in $\mathrm{Shed}(G)$;

\emph{(v) }there exist two disjoint maximum independent sets included in
$\mathrm{Shed}(G)$;

\emph{(vi) }$G=\alpha\left(  G\right)  K_{2}$.
\end{theorem}

\begin{proof}
\emph{(i) }$\Rightarrow$\emph{ (ii)} By Corollary \ref{cor5} and Theorem
\ref{lem1}.

\emph{(ii) }$\Rightarrow$\emph{ (iii) }By Theorem \ref{lem1}, every vertex of
$G$ is a leaf. Hence, $\mathrm{Shed}(G)=V\left(  G\right)  $.

\emph{(iii) }$\Rightarrow$\emph{ (iv)} Evident.

\emph{(iv) }$\Rightarrow$\emph{ (v)} By Corollary \ref{cor5}, $G$ has two
disjoint maximum independent sets.

\emph{(v) }$\Rightarrow$\emph{ (i)} Obvious.

\emph{(vi) }$\Leftrightarrow$\emph{ (ii) }By Theorem \ref{lem1}.
\end{proof}

\section{Trees}

Clearly, a leaf is a shedding vertex if and only if its unique neighbor is a
leaf as well. Hence, the only tree $T$ having $\mathrm{Shed}\left(  T\right)
=V\left(  T\right)  $ is $T=K_{2}$. Notice that $\mathrm{Shed}\left(
P_{4}\right)  =\left\{  v:\deg\left(  v\right)  =2\right\}  $, and no maximal
independent set of $P_{4}$ is included in $\mathrm{Shed}\left(  P_{4}\right)
$.

\begin{proposition}
\label{prop2}Let $T$ be a tree, which is not isomorphic to $K_{2}$, and let
$S$ be a maximal independent set. Then $S\subseteq\mathrm{Shed}\left(
T\right)  $ if and only if $S=\mathrm{Shed}\left(  T\right)  $.
\end{proposition}

\begin{proof}
The assertion is clearly true for $T=K_{1}$, as $\mathrm{Shed}\left(
K_{1}\right)  =\emptyset$. Assume that $T\neq K_{1}$.

By Theorem \ref{th4}, it follows that the shedding vertices of a tree are
exactly the neighbors of its leaves.

Let $S\subseteq\mathrm{Shed}\left(  T\right)  $ be a maximal independent set
in $T$. Thus each vertex in $V\left(  T\right)  -S$ has a neighbor in $S$.

Assume, to the contrary, that there is some $v\in\mathrm{Shed}\left(
T\right)  -S$. Hence, there must exist $vx,vu\in E\left(  T\right)  $ such
that $x$ is a leaf and $u\in S$. Consequently, we infer that $x\notin%
\mathrm{Shed}\left(  T\right)  $, and consequently, $S\cup\left\{  x\right\}
$ is an independent set larger than $S$, in contradiction with the maximality
of $S$. In conclusion, $S=\mathrm{Shed}\left(  T\right)  $.

The converse is evident.
\end{proof}

\begin{corollary}
\label{cor6}A tree $T$ has a maximum independent set consisting of only
shedding vertices if and only if $T=K_{2}$.
\end{corollary}

\begin{proof}
Clearly, $T\neq K_{1}$, since $\mathrm{Shed}\left(  K_{1}\right)  =\emptyset$.

Assume, on the contrary, that $T\neq K_{2}$ and let $S$ be a maximum
independent set such that $S\subseteq\mathrm{Shed}\left(  T\right)  $. Since
$\mathrm{Shed}\left(  K_{1}\right)  =\emptyset$, we infer that $V\left(
T\right)  \geq3$. By Proposition \ref{prop2}, we know that $S=\mathrm{Shed}%
\left(  T\right)  $.

Clearly, $\mathrm{Leaf}\left(  T\right)  $ is independent and $\left\vert
\mathrm{Leaf}\left(  T\right)  \right\vert \geq2$. By Theorem \ref{th4}, every
vertex of $S$ has a leaf as a neighbor. Consequently, $\mathrm{Leaf}\left(
T\right)  $ is a maximum independent set, because $\left\vert \mathrm{Leaf}%
\left(  T\right)  \right\vert \geq\left\vert S\right\vert =\alpha\left(
T\right)  $. Since all the vertices of $\mathrm{Leaf}\left(  T\right)  $ are
leaves of $T$, the subgraph $T-\mathrm{Leaf}\left(  T\right)  $ is a tree
containing all the vertices of $S$. Hence, there exists some $v\in V\left(
T-\mathrm{Leaf}\left(  T\right)  \right)  -S$, because $S$ is independent in
the tree $T-\mathrm{Leaf}\left(  T\right)  $ and $\left\vert S\right\vert
\geq2$. Hence, no neighbor of $v$ in $T$ is a leaf. Therefore, $\mathrm{Leaf}%
\left(  T\right)  \cup\left\{  v\right\}  $ is an independent set larger than
$\mathrm{Leaf}\left(  T\right)  $, in contradiction with $\mathrm{Leaf}\left(
T\right)  \in\Omega\left(  T\right)  $. In conclusion, $T=K_{2}$.

The converse is obvious.
\end{proof}

Notice that $\left\vert \mathrm{Shed}\left(  K_{1}\right)  \right\vert
=0=\alpha\left(  K_{1}\right)  -1$, while $\left\vert \mathrm{Shed}\left(
K_{2}\right)  \right\vert =2=\alpha\left(  K_{2}\right)  +1$.

\begin{proposition}
For a tree $T\neq K_{2}$ the following are true:

\emph{(i)} $\left\vert \mathrm{Shed}\left(  T\right)  \right\vert \leq
\alpha\left(  T\right)  $;

\emph{(ii)} if $\mathrm{Shed}\left(  T\right)  $ is independent, then
$\left\vert \mathrm{Shed}\left(  T\right)  \right\vert \leq\alpha\left(
T\right)  -1$.

Moreover, both inequalities are tight.
\end{proposition}

\begin{proof}
\emph{(i) }According to\emph{ }Theorem \ref{th4}, $\left\vert \mathrm{Shed}%
\left(  T\right)  \right\vert \leq\left\vert \mathrm{Leaf}\left(  T\right)
\right\vert $. In addition, $\left\vert \mathrm{Leaf}\left(  T\right)
\right\vert \leq\alpha\left(  T\right)  $, since $\mathrm{Leaf}\left(
T\right)  $ is independent. The series of graphs $P_{n}\circ K_{1},n\geq2$
shows that $\alpha\left(  T\right)  $ is the tight upper bound for $\left\vert
\mathrm{Shed}\left(  T\right)  \right\vert $.

\emph{(ii) }The inequality $\left\vert \mathrm{Shed}\left(  T\right)
\right\vert \leq\alpha\left(  T\right)  -1$ directly follows from Corollary
\ref{cor6}.

Let $p\geq2$, and $T$ be the tree obtained from $K_{1,p}$ by adding $p$
vertices, each one joined to a leaf of $K_{1,p}$. The set $\mathrm{Shed}%
\left(  T\right)  $ consists of all the leaves of $K_{1,p}$, while
$\alpha\left(  T\right)  =p+1$. Hence, $\left\vert \mathrm{Shed}\left(
T\right)  \right\vert =p=\alpha\left(  T\right)  -1$. Thus, in this case,
$\alpha\left(  T\right)  -1$ is the tight upper bound for $\left\vert
\mathrm{Shed}\left(  T\right)  \right\vert $.
\end{proof}

\section{Unicyclic graphs}

A graph $G$ is \textit{unicyclic} if it is connected and has a unique cycle,
which we denote by $C=\left(  V(C),E\left(  C\right)  \right)  $. Let
$N_{1}(C)=\{v:v\in V\left(  G\right)  -V(C),N(v)\cap V(C)\neq\emptyset\}$, and
$T_{x}=(V_{x},E_{x})$ be the maximum subtree of $G-xy$ containing $x$, where
$x\in N_{1}(C),y\in V(C)$.

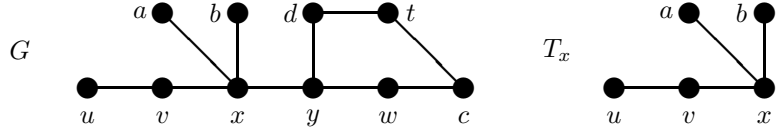
\begin{figure}[h]
\setlength{\unitlength}{1cm}\begin{picture}(5,1.8)\thicklines
\multiput(2,0.5)(1,0){6}{\circle*{0.29}}
\multiput(3,1.5)(1,0){4}{\circle*{0.29}}
\put(2,0.5){\line(1,0){5}}
\put(3,1.5){\line(1,-1){1}}
\put(4,0.5){\line(0,1){1}}
\put(5,1.5){\line(1,0){1}}
\put(6,1.5){\line(1,-1){1}}
\put(5,0.5){\line(0,1){1}}
\put(2,0.1){\makebox(0,0){$u$}}
\put(3,0.1){\makebox(0,0){$v$}}
\put(4,0.1){\makebox(0,0){$x$}}
\put(5,0.1){\makebox(0,0){$y$}}
\put(6,0.1){\makebox(0,0){$w$}}
\put(7,0.1){\makebox(0,0){$c$}}
\put(2.7,1.5){\makebox(0,0){$a$}}
\put(3.7,1.5){\makebox(0,0){$b$}}
\put(4.7,1.5){\makebox(0,0){$d$}}
\put(6.3,1.5){\makebox(0,0){$t$}}
\put(1.1,1){\makebox(0,0){$G$}}
\multiput(9,0.5)(1,0){3}{\circle*{0.29}}
\multiput(10,1.5)(1,0){2}{\circle*{0.29}}
\put(9,0.5){\line(1,0){2}}
\put(10,1.5){\line(1,-1){1}}
\put(11,0.5){\line(0,1){1}}
\put(9,0.1){\makebox(0,0){$u$}}
\put(10,0.1){\makebox(0,0){$v$}}
\put(11,0.1){\makebox(0,0){$x$}}
\put(9.7,1.5){\makebox(0,0){$a$}}
\put(10.7,1.5){\makebox(0,0){$b$}}
\put(8.25,1){\makebox(0,0){$T_{x}$}}
\end{picture}\caption{$G$ is a unicyclic non-K\"{o}nig-Egerv\'{a}ry graph with
$V(C)=\{y,d,t,c,w\}$.}%
\end{figure}

\begin{lemma}
\cite{LevMan2012} If $G$ is a unicyclic graph, then $n\left(  G\right)
-1\leq\alpha(G)+\mu(G)\leq n\left(  G\right)  $.
\end{lemma}

For instance, every $C_{2k}$ is a unicyclic K\"{o}nig-Egerv\'{a}ry graph,
while every $C_{2k+1}$ is a unicyclic non-K\"{o}nig-Egerv\'{a}ry graph.

\begin{theorem}
\cite{LevMan2012}\label{th13} Let $G$ be a unicyclic
non-K\"{o}nig-Egerv\'{a}ry graph. Then the following assertions are true:

\emph{(i) }$W\in\Omega\left(  T_{x}\right)  $ if and only if $W=S\cap V\left(
T_{x}\right)  $\ \ for some $S\in\Omega\left(  G\right)  $;

\emph{(ii)} $\mathrm{core}\left(  G\right)  =%
{\displaystyle\bigcup}
\left\{  \mathrm{core}\left(  T_{x}\right)  :x\in N_{1}\left(  C\right)
\right\}  $.
\end{theorem}

\begin{theorem}
\label{th10}A unicyclic graph $G$ has two disjoint maximum independent sets if
and only if, either $G$ is a bipartite graph with a perfect matching, or there
is a vertex $v$ belonging to its unique cycle, such that $G-v$ has a perfect matching.
\end{theorem}

\begin{proof}
Let $S_{1},S_{2}\in$ $\Omega\left(  G\right)  $ be such that $S_{1}\cap
S_{2}=\emptyset$. Clearly, $\mathrm{core}\left(  G\right)  =\emptyset$.

\textit{Case 1}. $G$ is a K\"{o}nig-Egerv\'{a}ry graph. Then, by Theorem
\ref{th9}, it follows that $G$ must a bipartite graph with a perfect matching.

\textit{Case 2}. $G$ is not a K\"{o}nig-Egerv\'{a}ry graph. Since
$\mathrm{core}\left(  G\right)  =\emptyset$, Theorem \ref{th13}\emph{(ii)}
implies $\mathrm{core}\left(  T_{x}\right)  =\emptyset$ for every $x\in
N_{1}\left(  C\right)  $. Hence, each $T_{x}$ has a perfect matching, by
Theorem \ref{th8}, and $S_{1}\cap V\left(  T_{x}\right)  ,S_{2}\cap V\left(
T_{x}\right)  \in\Omega\left(  T_{x}\right)  $, according to Theorem
\ref{th13}\emph{(i)}. Therefore, we get
\[
\left\vert S_{1}\cap V\left(  C\right)  \right\vert =\left\vert S_{2}\cap
V\left(  C\right)  \right\vert =\frac{\left\vert V\left(  C\right)
\right\vert -1}{2}.
\]
Thus, there is some $v\in V\left(  C\right)  $, such that $v\notin S_{1}\cup
S_{2}$. Finally, $G-v$ is a forest with a perfect matching.

Conversely, if $G$ is a bipartite graph with a perfect matching, then its
bipartition is comprised of two disjoint maximum independent sets.

Otherwise, there is a vertex $v$ belonging to its unique cycle, such that
$G-v$ is a forest with a perfect matching. If $\left\{  A,B\right\}  $ is a
bipartition of the vertex set of $G-v$, then clearly, $A$ and $B$ are disjoint
independent sets of $G$ of size equal to%
\[
\mu\left(  G\right)  =\mu\left(  G-v\right)  =\frac{n\left(  G\right)  -1}%
{2}.
\]
Moreover, $A,B\in\Omega\left(  G\right)  $, because, otherwise, we have
\[
n\left(  G\right)  -1-\mu\left(  G\right)  =\alpha\left(  G\right)
>\left\vert A\right\vert =\mu\left(  G-v\right)  =\mu\left(  G\right)  ,
\]
which leads to the following contradiction: $n\left(  G\right)  -1>2\mu\left(
G\right)  $.
\end{proof}

It is well-known that the matching number of a bipartite graph $G$ can be
computed in $O\left(  n\left(  G\right)  ^{\frac{5}{2}}\right)  $
\cite{HopfortKarp1973}. Thus Theorem \ref{th10} implies the following.

\begin{corollary}
One can decide in polynomial time whether a unicyclic graph has two disjoint
maximum independent sets.
\end{corollary}

\section{Conclusions}

Theorem \ref{th7} and Corollary \ref{cor1} provide us with a complete
description of very well-covered graphs of girth $\geq5$ containing a pair of
disjoint maximum independent sets. Corollary \ref{cor4} tells us that the only
girth under consideration left is four.

\begin{problem}
Find a constructive characterization of very well-covered graphs (bipartite
well-covered) of girth equal to $4$, that have two disjoint maximum
independent sets at least.
\end{problem}

The same question may be asked about other classes of graphs. Recall that $G$
is an \textit{edge }$\alpha$\textit{-critical graph} if $\alpha\left(
G-e\right)  >\alpha\left(  G\right)  $, for every $e\in E\left(  G\right)  $.
For instance, every odd cycle $C_{2k+1}$ and its complement are edge $\alpha
$-critical graphs. Moreover, both $C_{2k+1}$ and $\overline{C_{2k+1}}$\ have
two disjoint maximum independent sets.

\begin{conjecture}
\cite{JansenKruger2015} If $G$ is an edge $\alpha$-critical graph without
isolated vertices, then it has two disjoint maximum independent sets.
\end{conjecture}

It is known that the decision problem whether there are two disjoint maximal
independent sets in a graph is \textbf{NP}-complete \cite{HLR2009}.

\begin{conjecture}
It is \textbf{NP}-complete to recognize (well-covered) graphs with two
disjoint maximum independent sets.
\end{conjecture}

The friendship graph $F_{q}$ is a non-K\"{o}nig-Egerv\'{a}ry graph with
exactly $2^{\alpha\left(  G\right)  }$ maximum independent sets. Thus Theorem
\ref{th2} motivates the following.

\begin{problem}
Characterize non-K\"{o}nig-Egerv\'{a}ry graphs with $\left\vert \Omega
(G)\right\vert =2^{\alpha\left(  G\right)  }$.
\end{problem}

\section{Acknowledgements}

We express our gratitude to Isabel Beckenbach, who suggested a number of
remarks that helped us make proofs of some theorems clearer.

\end{document}